\documentclass[12pt,oneside,reqno]{amsart}

\usepackage{amssymb,amscd,stmaryrd}
\usepackage[mathcal]{eucal}
\usepackage{array,float}
\usepackage{enumitem}
\usepackage{tikz,mathtools}

\usepackage{xcolor}

\usepackage{xy}
\input xy
\xyoption{all}
\usepackage{pdflscape}
\usepackage{hyperref}

\usepackage{graphicx}

\usepackage[left=3cm, right=3cm]{geometry}
\usepackage{textcomp}
\numberwithin{equation}{section}
\numberwithin{equation}{subsection}

\newtheorem{thm}{Theorem}[section]
\newtheorem{proposition}[thm]{Proposition}

\newtheorem{lemma}[thm]{Lemma}

\newtheorem*{remark*}{Remark}

\newtheorem{example}[thm]{Example}

\newcommand{\Hom}{{\mathrm{Hom}}}
\newcommand{\res}{\mathrm{res}}
\newcommand{\tra}{{\mathrm{tra}}}

\newcommand{\Ho}{{\mathrm{H}}}

\newcommand{\Ker}{\mathrm{Ker}~}
\newcommand{\Image}{\mathrm{Im}~}
\newcommand{\B}{\mathrm{B}}

\newcommand{\bQ}{{\mathbb Q}}
\newcommand{\bZ}{{\mathbb Z}}

\makeatletter
\@namedef{subjclassname@2020}{\textup{2020} Mathematics Subject Classification}
\makeatother

\begin{document}

{

\title[On  Bogomolov multiplier of finite groups]{An exact sequence and triviality of Bogomolov multiplier of groups}

\author{Sumana Hatui}
\address{School of Mathematical Sciences, National Institute of Science Education and Research, An OCC of Homi Bhabha National Institute, Bhubaneswar 752050, Odisha, India}
 \email{sumana.iitg@gmail.com, sumanahatui@niser.ac.in}


\maketitle

\begin{abstract}
The Bogomolov multiplier $\B_0(G)$ of a finite group $G$ is  the subgroup of
the Schur multiplier  $\Ho^2(G,\mathbb Q/\mathbb Z)$ consisting of the cohomology classes which vanish after restricting to every  abelian subgroup of $G$.  
We give a new proof of a Hopf-type formula for $\B_0(G)$ and derive an exact sequence for the cohomological version of  the Bogomolov multiplier. Using this exact sequence we provide  necessary and sufficient conditions for the corresponding inflation homomorphism to be an epimorphism and to be the zero map.   
Finally, we  give a complete list  of  groups of order $p^6$, for odd prime $p$, having trivial Bogomolov multiplier, so completing the 2020 investigation of Chen and Ma.
\end{abstract}

\section{Introduction}
 The Bogomolov multiplier  $\B_0(G)$ of a finite group $G$ is the subgroup of $\Ho^2(G,\mathbb Q/\mathbb Z)$  containing all the cohomology classes whose restriction to every abelian subgroup of $G$ is zero.
It is a group theoretical invariant isomorphic to the unramified Brauer group of a given quotient space, and it has a connection with Noether’s problem on rationality of fields of invariants. Let  $G$ be a finite group, and let $V$ be a faithful representation of $G$ over the field $\mathbb C$. Then $G$ acts naturally on the field of rational functions $\mathbb C(V)$.
Noether’s problem \cite{noether} asks  for which groups $G$  the field of $G$-invariant functions $\mathbb C(V)^G$  is rational over $\mathbb C$?
A question related to the above is whether $\mathbb C(V)^G$  is stably rational over $\mathbb C$, i.e.,  are  there independent variables $y_1,y_2, \cdots, y_s$ such that  $\mathbb C(V)^G(y_1,y_2, \cdots, y_s)$ is a purely transcendental extension over $\mathbb C$?
This problem has close connection with Lüroth’s problem \cite{shaf} and the inverse Galois problem \cite{saltman, swan}. By Hilbert’s Theorem $90$ the stable rationality   does not depend upon the choice of $V$, it depends only on the group $G$.
Saltman \cite{saltman} introduced the unramified Brauer group $\Ho^2_{nr}(\mathbb C(V )^G,\mathbb Q/\mathbb Z) $ and using it he found examples of groups $G$ of order $p^9$ such that $\mathbb C(V)^G$  is not stably rational over $\mathbb C$. 
Later Bogomolov \cite{bogomolov} explored the group $\Ho^2_{nr}(\mathbb C(V )^G,\mathbb Q/\mathbb Z) $  and proved that it is canonically isomorphic to 
$$
\B_0(G)=\bigcap_{ A \leq G; A \text{ finite abelian}}\Ker \big(\res^G_A: \Ho^2(G,\mathbb Q/\mathbb Z) \to \Ho^2(A,\mathbb Q/\mathbb Z)\big),
$$
which is a subgroup of  $\Ho^2(G, \bQ/\bZ)$.   
Kunyavski{\u\i} \cite{Kun} called $\B_0(G)$  the \emph{Bogomolov multiplier} of $G$. 
Bogomolov \cite{bogomolov} used this description to give  examples of  groups $G$ of order $p^6$ having $\B_0(G) \neq 0$ and this  non-vanishing property of $\B_0(G)$ answers the classical Noether’s problem, i.e., the invariant field $\mathbb C(V)^G$ is not rational over $\mathbb C.$ So the triviality of $\B_0(G)$ is an obstruction to the Noether's problem.\\

Another description of  $\B_0(G)$ was given by Moravec  \cite{Mora}.
Let $G\wedge G$ denote  the \emph{non-abelian exterior square} of  group $G$  (see  Section \ref{Moravec} for details).  For $x,y\in G$, we write a commutator $[x,y]=x^{-1}y^{-1}xy$. 
Let $M(G)$ be the kernel of the commutator map  $G \wedge G \to [G,G]$.
Define $M_0(G)=\langle x \wedge y\mid [x,y]=1\rangle$.  
Moravec \cite{Mora}  defined 
$$\tilde{\B}_0(G) =M(G)/M_0(G)$$ and
proved that, for a finite group $G$, 
$$\B_0(G) \cong \Hom(\tilde{\B}_0(G), \mathbb{Q}/\mathbb{Z}).$$
Hence for finite groups $G$, $\B_0(G) \cong \tilde{\B}_0(G).$
This description  helps us to compute $\B_0(G)$ explicitly;  see for example \cite{ Mora1, Mora,  Mora2}.  We also use this description in Section \ref{Bogop6} to compute $\B_0(G)$ for  groups $G$ of order $p^6$, where $p$ is an odd prime.

\vspace{.2cm}

We establish some notation before proceeding  further.
We denote the commutator subgroup  and center of $G$ by $G'$ and $Z(G)$ respectively.  The set of commutators $\{[x,y] \mid x,y \in G\}$  is denoted by $K(G)$.
For subgroups $H$ and $K$ of $G$, the group  generated by 
$\{ [h,k], h\in H, k\in K \}$ is denoted by $[H,K]$.

\vspace{.2cm}

Let $G$ be a quotient of the free group $F$ by the normal closure  $R$ of a given set of relations  of $F$. In this context we say that $F/R$ {\it determines} a presentation for $G$.
Moravec  \cite{Mora}  proved a Hopf-type formula for $\tilde{\B}_0(G)$ by showing that 
\begin{eqnarray*}
\tilde{\B}_0(G) \cong   \frac{F' \cap R}{\langle K(F) \cap R\rangle}.
\end{eqnarray*} 
To define this isomorphism, he uses the notion of covering group and exterior square of the group $G$. 
Now we are ready to state our first main result  which provide an alternative proof of a Hopf-type formula. 
Our defined map is more explicit in nature and simplifies the proof of  \cite{Mora}.
This isomorphism will   also be  used  to prove Theorem \ref{exactsequence_Bogo}} and  Theorem \ref{inflation_epimorphism}.

\vspace{.2cm}

\begin{thm}\label{Hopftypeformula}
Let $G$ be a finite group with  presentation determined by $F/R$. If $A=\mathbb{Q}/\mathbb{Z}$, then the transgression map 
$$\tra: \Hom(R/\langle K(F) \cap R \rangle, A) \to \B_0(G),$$
corresponding to the exact sequence
$$
1 \to R/\langle K(F) \cap R \rangle \to  F/\langle K(F) \cap R \rangle  \to G \to 1,
$$
is surjective. Consequently,  
\begin{eqnarray*}
\B_0(G) \cong  \Hom \Big(\frac{F' \cap R}{\langle K(F) \cap R\rangle}, \mathbb{Q}/\mathbb{Z}\Big).
\end{eqnarray*}
\end{thm}
We prove this result in Section \ref{Proofof Theorems}. \\

In the next result, we provide an exact sequence for $\B_0(G)$.  
Let $G=F/R$ and let  $N=S/R$ be a normal subgroup of $G$. Our  aim is to define a homomorphism 
$$\theta: \B_0(G) \to  \Hom\big(\frac{R \cap \langle K(F) \cap S\rangle}{\langle K(F) \cap R \rangle} , \mathbb{Q}/\mathbb{Z}\big)$$ and  prove the exactness of the following sequence which is the  dual of the exact sequence  given in \cite[Proposition 3.12]{Mora}. 
\begin{thm}\label{exactsequence_Bogo}
Let $G$ be a finite group with  presentation determined by $F/R$, and let $N=S/R$ be a normal subgroup of $G$.  Suppose $\mathbb{Q}/\mathbb{Z}$ is a trivial $G$-module.  For a suitable map $\theta$, the following sequence is exact:

\begin{align}\label{sequence}\tag{1.1}
 0 \to  \Hom\big(\frac{N\cap G'}{\langle  N \cap K(G)\rangle} , \mathbb{Q}/\mathbb{Z}\big)  \xrightarrow{\tra}  \B_0(G/N)  
\xrightarrow{\inf} \B_0(G)  \xrightarrow{\theta} \\
 \Hom\big(\frac{R \cap \langle K(F) \cap S\rangle}{\langle K(F) \cap R \rangle} , \mathbb{Q}/\mathbb{Z}\big)   \to 0.  \nonumber
\end{align}
\end{thm}

\vspace{.2cm}

In \cite{verexact}, Vermani gave an exact sequence for the second cohomology group and using it he  provided necessary and sufficient conditions for the inflation homomorphism  to be an epimorphism or the zero map (for the second cohomology group).  We prove an analogous result
for the Bogomolov multiplier.
\vspace{.2cm}

\begin{thm}\label{inflation_epimorphism}
Let $G$ be a finite group with  presentation determined by $F/R$, and let $N=S/R$ be a normal subgroup of $G$.
\begin{enumerate}[label=(\roman*)]
\item The map $\inf: \B_0(G/N) \to \B_0(G)$ is an epimorphism if and only if 
$$R \cap \langle K(F) \cap S\rangle = \langle K(F) \cap R\rangle.$$

\item The map $\inf: \B_0(G/N) \to \B_0(G)$ is the zero map if and only if $$R \cap \langle K(F) \cap S\rangle =  F' \cap R.$$
\end{enumerate}
\end{thm}
\vspace{.2cm}
As an application of this result, we  show that, for extra-special  $p$-groups $G$, $\B_0(G)=0$; see Proposition \ref{extraspecial}, which leads to another proof of \cite[Proposition 2.1]{kangbogomolov}.

\vspace{.2cm}


In Theorem \ref{Centralproduct} we provide some necessary and sufficient conditions for the central product of groups to have trivial Bogomolov multiplier. In Theorem \ref{Heisenberg} we show that the  Bogomolov multiplier of every generalized discrete Heisenberg groups is trivial.
\vspace{.2cm}

Our final aim is to list all non-abelian groups $G$ of order $p^6$ for odd prime $p$,  having $\B_0(G)$ trivial. Groups of order $p^6$  were classified by James  \cite{James}. These  groups  are divided into isoclinism classes $\Phi_i, 1 \leq i \leq 43$. 
The isoclinism class $\Phi_1$ contains abelian groups and for finite abelian groups $G$, $\B_0(G)$ is trivial; see  \cite{swan}.
For groups $G$ of order $\leq p^4$, $\B_0(G)=0$; see \cite{bogomolov}.
The groups of order $p^5$ having trivial Bogomolov multiplier were found in \cite{Mora1}.
In \cite{chen}, the Bogomolov multipliers of some groups of order $p^6$ were recently studied, and with the help of a computer calculation, it has been conjectured that   $\B_0(\Phi_{15})=\B_0(\Phi_{28})=\B_0(\Phi_{29})=0$.
In the next result we confirm this conjecture  and  provide a complete list of  non-abelian groups of order $p^6$ for odd prime $p$, having a trivial Bogomolov multiplier.
\begin{thm}\label{groupsp6}
Let $G$ be a non-abelian group of order $p^6,$ where $p$ is an odd prime. Then $\B_0(G)=0$ if and only if  $G \in \Phi_j $ for $j \in \{2,3, \cdots, {43}  \}   \setminus    \{{10}, {18},{20}, {21},{36}, {38}, {39}\}$. 
\end{thm}
We prove this result in the last section.

\section{Preliminaries}\label{preliminaries}

Let $D$ be a divisible abelian group regarded as a trivial $G$-module. Consider an exact sequence
\[
1 \rightarrow N \rightarrow  G\rightarrow G/N \rightarrow 1.
\]
Then the Hochschild-Serre spectral sequence \cite{Hoch} for cohomology of groups yields the following exact sequence
\begin{equation*}
0 \to \Hom(G/N, D) \xrightarrow{\inf} \Hom(G, D) \xrightarrow{\res} \Hom( N,D) \xrightarrow{\tra} \Ho^2(G/N, D)  \xrightarrow{\inf} \Ho^2(G, D),
\end{equation*}
where $\tra:\Hom( N,D) \to \Ho^2(G/N, D) $ is the transgression homomorphism
 given by $f \mapsto \tra(f) = [\alpha]$, with
$$\alpha(\overline{x},\overline{y}) = f(\mu (\overline{x})\mu(\bar{y})\mu(\bar{xy})^{-1}), \,\, \text{for all}\,\,  \overline{x}, \overline{y} \in G/A, $$
for a section $\mu: G/N \rightarrow G$.
The inflation homomorphism, $\inf : \Ho^2(G/N,D)   \to  \Ho^2(G, D) $ is given by $[\alpha] \mapsto \inf([\alpha]) = [\beta]$, where $\beta(x,y) = \alpha(xN,yN)$, for all $x,y \in G$. 
The map $\res$ denotes restriction homomorphism.

\vspace{.2cm}

Now we check  that these  above maps can be defined for $\B_0(G)$ also.
First we define 
$$\res':\Hom(G, \mathbb{Q}/\mathbb{Z})    \to  \Hom(\frac{N}{\langle K(G) \cap N\rangle} , \mathbb{Q}/\mathbb{Z})  $$
by $\res'(f)(\bar{n})=f(n)$ for $f \in \Hom(G, \mathbb{Q}/\mathbb{Z})$ and $\bar{n} \in \frac{N}{\langle K(G) \cap N\rangle}$.
It is easy to see that $\res'$ is a well defined homomorphism.

Next we define the transgression map $$\tra:\Hom\big(\frac{N}{\langle K(G) \cap N\rangle} , \mathbb{Q}/\mathbb{Z}\big)  \to  \B_0(G/N)$$ 
given by $\chi \mapsto \tra(\chi) = [\alpha]$, with
$$\alpha(\overline{x},\overline{y}) = 
\chi(\overline{\mu (\overline{x})\mu(\bar{y})\mu(\bar{xy})^{-1}}), \,\, \text{for all}\,\,  \overline{x}, \overline{y} \in G/N, $$
for a section $\mu: G/N \rightarrow G$.
If $\chi \in \Hom(\frac{N}{\langle K(G) \cap N\rangle}, \bQ/\bZ)$, then we show that $\tra(\chi) \in \B_0(G/N)$. 
Let $A/N$ be a  finite abelian subgroup of $G/N$ and let $\mu: G/N \to G$ be a section. For $\bar{a}, \bar{b} \in A/N$,
\begin{eqnarray*}
\tra(\chi)(\bar{a}, \bar{b})&=&\chi(\overline{\mu(\bar{a})\mu(\bar{b})\mu(\bar{a}\bar{b})^{-1}})=\chi\big(\overline{[\mu(\bar{a})^{-1}, \mu(\bar{b})^{-1}].  \mu(\bar{b})\mu(\bar{a})\mu(\bar{b}\bar{a})^{-1}} \big)
\\
&=&\chi(\overline{[\mu(\bar{a})^{-1}, \mu(\bar{b})^{-1}]})  \chi(\overline{\mu(\bar{b})\mu(\bar{a})\mu(\bar{b}\bar{a})^{-1}} )
\\
&=& \chi(\overline{\mu(\bar{b})\mu(\bar{a})\mu(\bar{b}\bar{a})^{-1}})  \text{ as } [\mu(\bar{a})^{-1}, \mu(\bar{b})^{-1}]\in \langle K(G) \cap N\rangle\\
&=&\tra(\chi)(\bar{b}, \bar{a}) .
\end{eqnarray*}
Thus, $\tra(\chi)$ is a symmetric $2$-cocycle on finite abelian groups $A/N$. Since $\mathbb{Q}/\mathbb{Z}$   is a divisible abelian group,  $[\tra(\chi)|_{A/N \times A/N}]=1$.
Thus $\tra $ is a well-defined homomorphism.
It is also easy to see that the inflation map $\inf:\B_0(G/N) \to \B_0(G)$ is a well defined homomorphism.

The following result describes the Bogomolov multiplier for the direct product of
groups.
\begin{thm}$($\cite[Theorem 1.4]{kang}$)$\label{Directproduct}
Let $G_1$ and $G_2$ be  finite groups.   The restriction map  
$\res: \B_0(G_1 \times G_2) \rightarrow \B_0(G_1) \times \B_0(G_2)$ is an isomorphism. 
\end{thm}

\subsection{Another description of $\B_0(G)$}\label{Moravec}
The \emph{non-abelian exterior square} of a group $G$, denoted by $G \wedge G$, is a  group generated by the symbols $x \wedge y$, $x,y \in G$ with the following relations:
$$xy \wedge z=(y^{-1}  xy\wedge y^{-1} zy)(y \wedge z),$$
$$x \wedge yz=(x \wedge z)(z^{-1} xz\wedge z^{-1} yz), $$
$$x \wedge  x=1.$$
It follows from the definition that the map  $f: G\wedge G \rightarrow G'$, defined on the generators by $f(g \wedge h)=[g,h]$, is an epimorphism. Let $M(G)=\ker f$.
Miller \cite{Miller}  proved that $M(G)$ is isomorphic to the second homology group $\Ho_2(G,\mathbb Z)$.  
We define $M_0(G)$, a subgroup of $M(G)$, generated by the set 
$\{ x \wedge y \mid [x,y]=1\}$.
We refer the reader to \cite{brown, Miller}  for more details.

\vspace{.2cm}

A different description of $G \wedge G$ was  introduced in \cite{rocco}, which is  more useful for evaluating the exterior square of  group $G$.
By $G^{\phi}$ we denote  the isomorphic copy of  group $G$ via the isomorphism $\phi$. Consider the group
$$\tau (G):=\langle G,G^{\phi} \mid \Re, \Re ^{\phi}, [g_1,g_2^{\phi}]^g=[g_1^g, (g_2^g)^{\phi}]=[g_1,g_2^\phi]^{g^\phi}, [g,g^\phi]=1\;\text{for all}\;  g,g_1,g_2 \in G\rangle$$
where $\Re, {\Re}^\phi$ are the defining relations of $G$ and $G^{\phi}$ respectively.    It follows from \cite[Proposition 16]{blythmorse} that the map $\Phi: G \wedge G \rightarrow [G,G^\phi]$ defined by 
$$\Phi(g  \wedge h)=[g,h^\phi],\;\; g,h \in G$$ 
is an isomorphism. 
Let $\eta:=f \circ\Phi^{-1}: [G,G^\phi] \to G'$ and $M^*(G)=\ker \eta=\Phi(M(G))$, 
$M_0^*(G)=\Phi(M_0(G))$.
It follows from \cite{Mora} that 
$$\B_0(G) \cong M^*(G)/M_0^*(G).$$
This will be used for computing $\B_0(G)$ in Section \ref{Bogop6}.

\vspace{.2cm}

\begin{lemma}$($See \cite{blythmorse}$)$\label{lemma1}
For a group $G$, the following properties hold in $\tau (G)$.
\begin{enumerate}[label=(\roman*)]
\item If $G$ is nilpotency class $c$, then $\tau(G)$ has nilpotency class at most $c+1$.
\item If $G'$ is of nilpotency class $c$, then $[G,G^\phi]$ has nilpotency class $c$  or $c+1$.
\item If $G$ has nilpotency class $\leq 2$, then $[G,G^\phi]$ is abelian.
\item $[g_1^\phi,g_2,g_3]=[g_1,g_2^\phi,g_3]=[g_1,g_2,g_3^\phi]=[g_1^\phi,g_2^\phi,
g_3]=[g_1^\phi,g_2,g_3^\phi]=[g_1,g_2^\phi,g_3^\phi]$ for all $g_1,g_2,g_3 \in G$.
\item $[[g_1,g_2^\phi],[h_1,h_2^\phi]]=[[g_1,g_2],[h_1,h_2]^\phi]$ for all $g_1,g_2,h_1,h_2 \in G$.
\item $[[g_1,g_2^{\phi}],[g_2,g_1^{\phi}]]=1$ for all $g_1,g_2\in G$.
\item $[g_1,g_2^\phi]=[g_1^\phi,g_2]$ for all $g_1,g_2 \in G$.
\end{enumerate}
\end{lemma}

We  define $[x_1, x_2, \ldots, x_n]=[[x_1, x_2, \ldots, x_{n-1}], x_n]$, for $x_1,x_2, \ldots, x_n \in G, n \geq 2$.

\begin{lemma}$($ \cite[Lemma 3.7]{Mora1} $)$\label{lemma2}
Let $G$ be a nilpotent group of class $\leq 5$, and let $n$ be a positive integer.
For $x,y \in G$,
$$
[x^n,y]=[x,y]^n[x,y,x]^{n \choose 2}[x,y,x,x]^{n \choose 3} [x,y,x,x,x]^{n \choose 4} [x,y,x,[x,y]]^{\sigma(n)},
$$
where $\sigma(n)=\frac{n(n-1)(2n-1)}{6}.$
\end{lemma}

\begin{lemma}$($\cite[Lemma 2.2]{Mora1}  $)$\label{lemma3}
Let $G$ be a finite polycyclic group  with polycyclic generating sequence $\{g_1,g_2,\ldots , g_n\}$. Then the group $[G,G^\phi]$, a subgroup of $\tau(G)$, is generated by the set $\{[g_i, g_j^\phi], i,j=1,2,\ldots,n, i>j\}$.
\end{lemma}

\section{Exact sequences}\label{Proofof Theorems}
In this section we prove Theorem \ref{Hopftypeformula}, Theorem \ref{exactsequence_Bogo} and Theorem \ref{inflation_epimorphism}.
Following the proof of \cite[Theorem 11.9.1]{vermani}, it is easy to check that the following sequence is exact:
\begin{align}\label{Bogo_equation}\tag{3.1}
0 \to \Hom(G/N, \mathbb{Q}/\mathbb{Z}) \xrightarrow{\inf}  \Hom(G, \mathbb{Q}/\mathbb{Z})    \xrightarrow{\res'}  \Hom\big(\frac{N}{\langle K(G) \cap N\rangle} , \mathbb{Q}/\mathbb{Z}\big)  \\
  \xrightarrow{\tra}  \B_0(G/N) \xrightarrow{\inf} \B_0(G).  \nonumber
\end{align}

\vspace{.2cm}

\begin{lemma}\label{Trans}
Let $1 \to N \to G \to G/N \to 1$ be an exact sequence. The image of the corresponding transgression map 
$$
\tra: \Hom(N/\langle K(G) \cap N \rangle, \mathbb{Q}/\mathbb{Z}) \to \B_0(G/N)
$$
is isomorphic to  $\Hom\big(\frac{G'\cap N}{\langle K(G) \cap N \rangle}, \mathbb{Q}/\mathbb{Z}\big)$.
\end{lemma}
\begin{proof}
We have a surjective homomorphism 
$$\eta:\Hom\big(\frac{N}{\langle K(G) \cap N\rangle} , \mathbb{Q}/\mathbb{Z}\big) \to \Hom\big(\frac{N\cap[G,G]}{\langle K(G) \cap N\rangle} , \mathbb{Q}/\mathbb{Z}\big)$$ induced by the inclusion map $N\cap[G,G] \hookrightarrow N$. 
Let $$K=\{f \in \Hom\big(\frac{N}{\langle K(G) \cap N\rangle} , \mathbb{Q}/\mathbb{Z}\big) \mid f \text{ can be extended to a homomorphism } G \to  \mathbb{Q}/\mathbb{Z}\}.$$ 
By the exact sequence (\ref{Bogo_equation}),  $K= \mathrm{Im}(\res')=\ker(\tra)$. Now we show that $K=\ker \eta$.  It will imply that
\[
\mathrm{Im}(\tra) \cong 
\frac{\Hom\big(\frac{N}{\langle K(G) \cap N\rangle} , \mathbb{Q}/\mathbb{Z}\big) }{K}
\cong \Hom\big(\frac{N\cap G'}{\langle K(G) \cap N\rangle} , \mathbb{Q}/\mathbb{Z}\big).
\]

Let $f \in \Hom\big(\frac{N}{\langle K(G) \cap N\rangle} , \mathbb{Q}/\mathbb{Z}\big)$ such that $f \in K$. Then $\frac{N\cap [G,G]}{\langle K(G) \cap N\rangle} \subseteq \ker f$ and so $f \in \ker \eta$. 
Conversely, let  $g \in \ker \eta$. Then $g$ can be viewed as a homomorphism $\frac{N}{N \cap[G,G]} \to \mathbb{Q}/\mathbb{Z}$.  As 
$\frac{N}{N \cap[G,G]} \cong \frac{N[G,G]}{[G,G]}$ and  $\mathbb{Q}/\mathbb{Z}$ is a divisible abelian group, $g$ can be extended to a homomorphism  $G/[G,G] \to \mathbb{Q}/\mathbb{Z}$. Hence $g \in K$. Hence $K=\ker \eta$.
\end{proof}

\bigskip

Recall from \cite{moravec2012unramified} that a central extension $1   \to A\to H \to G \to 1$ is a {\it central CP extension} of $G$ by $A$ if every commuting pair of elements of $G$ lifts to a  commuting pair in $H$.\\

\noindent \textbf{Proof of Theorem \ref{Hopftypeformula}.}

\begin{proof}
By \cite[Theorem 3.1 and Proposition 2.4]{Urban}, there is a bijective correspondence between $\B_0(G)$ and central CP extensions of $G$ by $A=\mathbb{Q}/\mathbb{Z}$.
Let  $\alpha$ be a $2$-cocycle of $G$ such that $[\alpha] \in \B_0(G)$.  Then there is a   central CP extension
$$ 1 \to A \to  H \xrightarrow{\mu} G \to  1$$  corresponding to $\alpha$.
Since $F$ is a free group, there is a  homomorphism $f:F\to H$
such that  the following is commutative. Here   $\bar{f}$  denotes  the restriction of $f$ on $R$. 
 \[
 \xymatrix{
 1 \ar[r] & R \ar[r] \ar[d]^{\bar{f}} & F \ar[r]^{\pi}    \ar[d]^{f}  & G \ar[d]^{id}  \ar[r] &1 \\
 1 \ar[r] & A \ar[r] &  H \ar[r]^{\mu}    &  G \ar[r]&  1 \\
 }
 \] 
Let $[a,b] \in K(F) \cap R$. Then $\pi([a,b])=[\pi(a),\pi(b)]=\mu([f(a),f(b)])=1$ and so $[f(a),f(b)] \in  K(H)\cap A$.
By \cite[Proposition 3.3]{Urban}, it follows that $[f(a),f(b)]=f([a,b]) =1$. 
Hence $\langle K(F) \cap R \rangle \subseteq \ker f$ and the commutative diagram in Figure 1 is induced from  that given above. 
 \[
 \xymatrix{
 1 \ar[r] & R/\langle K(F) \cap R \rangle \ar[r] \ar[d]^{\bar{f}} & F/\langle K(F) \cap R \rangle \ar[r]^{\;\;\;\;\;\;\;\;\;\;\;\pi}    \ar[d]^{f}  & G \ar[d]^{id}  \ar[r] &1 \\
 1 \ar[r] & A \ar[r] &  H \ar[r]^{\mu}    &  G \ar[r]&  1 \\
 }
 \] 
 \[\text{Figure 1}\] \\
Now we show that $\tra(\bar{f})=[\alpha].$ 
Let $s$ be a section of $\pi$. Then  $f \circ  s$ is a section of $\mu$. Thus $\alpha$ is cohomologous to a cocycle 
$$\beta(x,y)=f(s(x)s(y)s(xy)^{-1})=\bar{f}(s(x)s(y)s(xy)^{-1})=\tra(\bar{f}).$$
Hence the map $\tra: \Hom(R/\langle K(F) \cap R \rangle, A) \to \B_0(G) $ is surjective.
Now by Lemma \ref{Trans},   
$$\B_0(G) \cong\Hom(\frac{F'\cap R}{\langle K(F) \cap R \rangle}, \mathbb{Q}/\mathbb{Z})  .$$
\end{proof}

\begin{thm}\label{longexact}
Let $F/R$ be a finite  presentation of a finite group $G$, let $N=S/R$ be a normal subgroup of $G$ and $\mathbb{Q}/\mathbb{Z}$ is a trivial $G$-module. For a suitable map $\theta$, the following sequence is exact:
\begin{eqnarray*}
0 \to \Hom(G/N, \mathbb{Q}/\mathbb{Z}) \xrightarrow{\inf}  \Hom(G, \mathbb{Q}/\mathbb{Z})    \xrightarrow{\res'}  \Hom\big(\frac{N}{\langle K(G) \cap N\rangle} , \mathbb{Q}/\mathbb{Z}\big)    \\
\xrightarrow{\tra}  \B_0(G/N) \xrightarrow{\inf} \B_0(G)  \xrightarrow{\theta}  \Hom\big(\frac{R \cap \langle K(F) \cap S\rangle}{\langle K(F) \cap R \rangle} , \mathbb{Q}/\mathbb{Z}\big)   \to 0.
\end{eqnarray*}
\end{thm}

\begin{proof}
\bigskip
Our  first goal is to define the map $\theta: \B_0(G)  \to \Hom\big(\frac{R \cap \langle K(F) \cap S\rangle}{\langle K(F) \cap R \rangle} , \mathbb{Q}/\mathbb{Z}\big)$.
Let $$\psi: \Hom(\frac{S}{\langle K(F) \cap S \rangle},\bQ/\bZ)  \to  \Hom(\frac{R}{\langle K(F) \cap R \rangle},\bQ/\bZ) $$ be the map induced by the inclusion map 
$ R \to S.$ 
Consider the restriction homomorphism $$\delta : \Hom(\frac{R}{\langle K(F) \cap R \rangle},\bQ/\bZ)  \to \Hom(\frac{R \cap \langle K(F) \cap S \rangle}{\langle K(F) \cap R \rangle} ,\bQ/\bZ). $$ 
Since $\bQ/\bZ$ is a divisible abelian group and $\frac{R}{\langle K(F) \cap R \rangle}$ is an abelian group, every homomorphism  $\frac{R \cap \langle K(F) \cap S \rangle}{\langle K(F) \cap R \rangle} \to \bQ/\bZ$ can be extended to a homomorphism  
$$\frac{R }{\langle K(F) \cap R \rangle}\to \bQ/\bZ$$
and thus $\delta$ is an epimorphism.  It is easy to check that $\Image{\psi} = \Ker \delta$,
so
\begin{align}\label{Impsi}\tag{3.2}
\frac{\Hom(\frac{R}{\langle K(F) \cap R \rangle},\bQ/\bZ) }{\Image(\psi)} \cong \Hom(\frac{R \cap \langle K(F) \cap S \rangle}{\langle K(F) \cap R \rangle} ,\bQ/\bZ).
\end{align}


Now by Theorem \ref{Hopftypeformula} and  by (\ref{Bogo_equation}), we have the  commutative diagram in Figure 2, where each row is an exact sequence:\\
\[
\xymatrix{
	\Hom(  \frac{F}{\langle K(F) \cap S \rangle},\bQ/\bZ)  \ar[r]^{\;\; \res' } \ar[d] & \Hom(\frac{S}{\langle K(F) \cap S \rangle},\bQ/\bZ)  \ar[r]^{\;\;\;\;\;\;\tra}    \ar[d]^{\psi}  & \B_0(G/N)  \ar[d]^{\inf} \to 0 \\
	\Hom(  \frac{F}{\langle K(F) \cap R \rangle} ,\bQ/\bZ) \ar[r]^{\;\; \res' } &  \Hom(\frac{R}{\langle K(F) \cap R \rangle},\bQ/\bZ)  \ar[r]^{\;\;\;\;\;\;\tra}    &  \B_0(G)  \to 0.  \\
}
\]  \[\text{Figure 2}\] \\
Figure 2 implies  that  the map $\phi: \frac{\Hom(\frac{R}{\langle K(F) \cap R \rangle},\bQ/\bZ)}{\Image(\res ')} \to   \B_0(G)$  given by 
\[
\phi(f+\Image({\res')})=\tra(f)
\]
is an isomorphism.

Observe that $\Image(\res') \subseteq \Image(\psi)$. Hence, we have the natural projection map 
\[
\tau:  \frac{\Hom(\frac{R}{\langle K(F) \cap R \rangle},\bQ/\bZ)}{\Image(\res')}  \to  \frac{\Hom(\frac{R}{\langle K(F) \cap R \rangle},\bQ/\bZ)}{\Image(\psi)}. 
\]
Thus we define the map 
\[
\theta:  \B_0(G) \to  \frac{\Hom(\frac{R}{\langle K(F) \cap R \rangle},\bQ/\bZ) }{\Image(\psi)}
\]
by $\theta=\tau \circ \phi^{-1}$.  
Since $\tau$ is an epimorphism, $\theta$ is also an epimorphism.

Now we prove that $\mathrm{Im}(\inf)=\ker \theta$.
Let $[\alpha] \in \B_0(G)$ be such that $\theta([\alpha])=0$. Then there is $f \in \Hom\big(\frac{R}{\langle K(F) \cap R \rangle}, \bQ/\bZ\big) \cap \mathrm{Im}(\psi)$ such that  $$\phi(f+\mathrm{Im}(\res'))=\tra(f)=[\alpha].$$
Thus there exists $g \in \Hom\big(\frac{S}{\langle K(F) \cap S \rangle}, \bQ/\bZ)$ such that $\psi(g)=f$ and by the commutativity of Figure 2, 
$$[\alpha]=\tra(f)=\tra(\psi(g))=\inf (\tra(g)).$$ Hence  $[\alpha] \in \mathrm{Im}(\inf: \B_0(G/N) \to \B_0(G))$.
Conversely, assume $$[\alpha] \in \mathrm{Im}(\inf: \B_0(G/N) \to \B_0(G)).$$
Then  there exists  $h \in \Hom\big(\frac{S}{\langle K(F) \cap S \rangle}, \bQ/\bZ\big)$ such that
$$[\alpha]=\tra(\psi(h))  \text{ and }\phi\big(\psi(h)+\Image(\res') \big)=[\alpha].$$ Hence
$\theta([\alpha])=\tau \circ \phi^{-1}([\alpha])=\tau(\psi(h)+\mathrm{Im}(\res'))=\psi(h)+\mathrm{Im}(\psi)=0$. 
Therefore $\mathrm{Im}(\inf)=\ker \theta$.
Now by exact sequence (\ref{Bogo_equation}) and by (\ref{Impsi}), our result follows.
\end{proof}

\noindent \textbf{Proof of Theorem \ref{exactsequence_Bogo}.}
\begin{proof}
Result follows from  Lemma \ref{Trans} and Theorem \ref{longexact}.
\end{proof}

\noindent \textbf{Proof of Theorem \ref{inflation_epimorphism}.}

\begin{proof}
$(i)$ By the  exact sequence (\ref{sequence}),    $\inf: \B_0(G/N) \to \B_0(G)$ is an epimorphism if and only if   $$\theta:\B_0(G)\to \Hom(\frac{R \cap \langle K(F) \cap S \rangle}{\langle K(F) \cap R \rangle} ,\bQ/\bZ)$$ 
is the zero map. Since $\theta$ is surjective,   $R \cap \langle K(F) \cap S \rangle=\langle K(F) \cap R \rangle$.

\vspace{.5cm}

$(ii)$   Let $\alpha \in \B_0(G)$. By   Figure 1 and by the proof of Theorem \ref{Hopftypeformula}, there exists  
$\bar{f} \in \Hom(\frac{R}{\langle K(F) \cap R \rangle}, \bQ/\bZ)$ such that  $\tra(\bar{f})=[\alpha]$,
and
the  isomorphism
$$\gamma:\B_0(G) \to \Hom\big(\frac{F' \cap R}{\langle K(F) \cap R \rangle}, \bQ/\bZ\big)$$ is   given by 
$\gamma(\alpha)=\bar{f}|_{\frac{F'\cap R}{\langle K(F) \cap R \rangle}}$.

Hence, according to the proof of  Theorem \ref{longexact}, we see that
$\phi(\bar{f}+\mathrm{Im}(\res'))=\alpha$ and $\theta(\alpha)=\bar{f}+\mathrm{Im}(\psi)$.
Thus, considering the isomorphism in (\ref{Impsi}), we obtain
$$\theta:\B_0(G)\to \Hom(\frac{R \cap \langle K(F) \cap S \rangle}{\langle K(F) \cap R \rangle} ,\bQ/\bZ)$$   given by
$\theta(\alpha)=\bar{f}|_{\frac{R \cap \langle K(F) \cap S \rangle}{\langle K(F) \cap R \rangle}}.$

Thus identifying  $\B_0(G)$ with  $\Hom\big(\frac{F' \cap R}{\langle K(F) \cap R \rangle}, \bQ/\bZ\big)$ via $\gamma$, we see that
\begin{eqnarray*}
\theta: \Hom(\frac{F' \cap R}{\langle K(F) \cap R \rangle}, \bQ/\bZ) \to \Hom(\frac{R \cap \langle K(F) \cap S \rangle}{\langle K(F) \cap R \rangle} ,\bQ/\bZ)
\end{eqnarray*}
is the restriction homomorphism. Hence, by the  exact sequence (\ref{sequence}),
 $\inf$ is the zero map if and only if  the restriction map $\theta$
 is an isomorphism, i.e.,  $F'\cap R=R \cap \langle K(F) \cap S\rangle$.
\end{proof}

The following was proved in  \cite[Proposition 2.1]{kangbogomolov} and \cite[Corollary 3.2]{michailov}.
Here we give a different proof, as an application of Theorem \ref{inflation_epimorphism}.

\begin{proposition}\label{extraspecial}
If $G$ is an extra-special $p$-group  of order $p^{2n+1},$  for $n\geq 1$, then $\B_0(G)=0$.
\end{proposition}

\begin{proof}
Let $G$ be an extra-special $p$-group of order $p^{2n+1}$. Then there is an extra-special $p$-group of exponent $p$ in the isoclinism class of $G$. Using \cite{Mora}, it is enough to prove our result for this group of exponent $p$.

First  consider the Heisenberg group of order $p^3$, say $H=F/R$, where $F=\langle x, y \rangle$ is the free group generated by two elements and $R$ is the normal closure in $F$ of $\{\gamma_3(F), x^p,y^p\}$. 
We deduce that
\begin{eqnarray*}
[x^iy^j, x^ry^s]&=&[x^i, x^ry^s] [y^j, x^ry^s]=[x^i,y^s][y^j, x^r ],\\
 &=& [x,y]^{is-jr}=[x^{is-jr},y] \pmod{\langle K(F) \cap R\rangle}.
\end{eqnarray*}
Hence, 
$$F'=K(F) \pmod{\langle K(F) \cap R\rangle}.$$ 
Therefore, $F' \cap R=\langle K(F) \cap R \rangle$. If   $N=S/R$ is a  normal subgroup of $G$, then  $F'\cap R=\langle K(F)\cap R \rangle =R\cap\langle K(F)\cap S\rangle$ as $\langle K(F)\cap R \rangle \subseteq R\cap\langle K(F)\cap S\rangle  \subseteq F'\cap R$.  Theorem \ref{inflation_epimorphism}(i--ii) implies that inf $: B_0(H/N) \to B_0(H)$ is both surjective and the zero map from $B_0(H/N)$  to $B_0(H)$. So $B_0(H)=0$.

Now we use induction to prove our result. 
Consider an extra-special $p$-group $G$ of exponent $p$ and order $p^{2n+1}, n>1$. Then $G$ is a central product of $H_1$ and $H_2$ such that
$H_1$ is an extra-special $p$-group of order $p^{2n-1}$ of exponent $p$ and $H_2=H$. 
Let $G=F/R$ and $H_1=S_1/R, H_2=S_2/R$. Then $F=S_1S_2$ and $[S_1,S_2] \subseteq R$. 
Take $N=\frac{(S_1'\cap S_2')R}{R}=S/R$, a normal subgroup of $G$.
By the induction hypothesis, 
\begin{align}\label{11}\tag{3.3}
S_1'= K(S_1) \pmod{\langle K(S_1) \cap R\rangle}
\end{align}
\begin{align}\label{22}\tag{3.4}
S_2'= K(S_2)  \pmod{\langle K(S_2) \cap R\rangle}.
\end{align}
Let   $f_1,f_2 \in F, s_1,s_1', s_{11}, s_{11}' \in S_1,$ $ s_2,s_2',s_{22},s_{22}' \in S_2$ such that
  $f_1=s_1s_2, f_2=s_1's_2', f_3=s_{11}s_{22}, f_4=s_{11}'s_{22}'$.
  
Now by \eqref{11} and \eqref{22}, 
$[s_1,s_1'][s_{11},s_{11}']=  [s_{111},s_{111}'] ~\text{and} ~
[s_2,s_2'][s_{22},s_{22}'] = [s_{222},s_{222}']$ for some $s_{111}, s_{111}' \in S_1,$  and $s_{222},s_{222}' \in S_2$,
so
\begin{eqnarray*}
[f_1,f_2][f_3,f_4]&&= [s_1s_2, s_1's_2'][s_{11}s_{22}, s_{11}'s_{22}']\\
&&= [s_1,s_1'][s_2,s_2'] [s_{11},s_{11}'][s_{22},s_{22}']  \pmod{\langle K(F) \cap R\rangle}\\
&&= [s_{111},s_{111}'][s_{222},s_{222}']=[s_{111}s_{222},s_{111}'s_{222}']  \pmod{\langle K(F) \cap R\rangle}
\end{eqnarray*}
Therefore,  $F'=K(F) \pmod{\langle K(F) \cap R\rangle}$ and so
$F'\cap R=\langle K(F)\cap R \rangle =R\cap\langle K(F)\cap S\rangle$.
Then by Theorem \ref{inflation_epimorphism}(i--ii),  inf $: B_0(G/N) \to B_0(G)$ is both surjective and the zero map from $B_0(G/N)$ to $B_0(G)$. So $B_0(G)=0$.
\end{proof}

\section{Triviality of $\B_0(G)$ for central products} \label{Bogocentral}
A group $G$ is a central product of  groups $H$ and $N$  if there exist $H_1 \leq Z(H)$,  $N_1 \leq Z(N)$ and an isomorphism $\xi: H_1 \to N_1$ such that $G \cong \frac{H \times N}{Z}$, where $Z=\{(a, \xi(a)^{-1})\mid a \in H_1\}$.
Kang and Kunyavski{\u\i} \cite{kangbogomolov} asked the following.\\

\noindent
\textbf{Question 1:} $($\cite[Question 3.3(i)]{kangbogomolov}$)$
Let $G$ be a central product of groups $H$ and $N$ such that  $\B_0(H)=\B_0(N)=0$. Is it true that $\B_0(G)=0$?\\

\noindent 
Rai  gave a negative answer to this question in \cite[Proof of Theorem 1.1]{Rai}.
As an approach to this question,  Michailov \cite[Theorem 3.1]{michailov} proved that 
if  $\eta: H \to N$ is a homomorphism such that  $\xi=\eta|_{H_1}$ is an isomorphism and $\B_0(H)=\B_0(N)=0$, then $\B_0(H/H_1)=0$ implies $\B_0(G)=0$. 
We establish the following and so generalize his result.

\begin{thm}\label{Centralproduct}
	Let $G$ be a central product of groups $H$ and $N$ via the map $\xi$ such that 
	$\B_0(H)=\B_0(N)=0$. Suppose any one of the following conditions is satisfied.
	\begin{enumerate}[label=(\roman*)]
		\item  $\xi(H_1\cap H') \subseteq  N_1\cap K(N)$. 
		\item There is a homomorphism $\eta: H \to N$ such that $\xi=\eta|_{H_1}$.
	\end{enumerate}
	Then $\B_0(G)=0$ if and only if $B_0(H/H_1)=0$.
\end{thm}

\begin{proof}
Let $G=\frac{H \times N}{Z}$, where $Z=\{(a, \xi(a)^{-1})\mid a \in H_1\}.$ 
It  follows from Theorem \ref{Directproduct} that
 $$\B_0(H\times N)=\B_0(H) \times \B_0(N)=0.$$
By (\ref{sequence}), we have the following exact sequence
\begin{align*}
 0 \to  \Hom\big(\frac{Z\cap(H' \times N')}{\langle Z \cap( K(H) \times K(N))\rangle} , \mathbb{Q}/\mathbb{Z}\big)  \xrightarrow{\tra}  \B_0(G)  
\xrightarrow{\inf} \B_0(H) \times \B_0(N)=0,  
\end{align*}
which implies that  $$\B_0(G) \cong \frac{Z\cap(H' \times N')}{\langle Z \cap( K(H) \times K(N))\rangle}.$$\\

$(i)$   Suppose $\B_0(G)=0$. Then   $Z\cap(H' \times N')=\langle Z \cap( K(H) \times K(N))\rangle $.
Let $a \in H_1 \cap H'$.
Then  $\xi(a)\in  N_1\cap K(N)  \subseteq N_1\cap N'.$
Therefore, $$(a, \xi(a)^{-1}) \in Z\cap(H' \times N')=\langle Z \cap( K(H) \times K(N))\rangle.$$ 
Hence $a \in \langle H_1 \cap K(H)  \rangle$ and so  $ H_1 \cap H'=\langle H_1 \cap K(H)  \rangle$. 
Thus $\B_0(H/H_1)=0$ follows from  the exact sequence  (\ref{sequence}),  \\

Conversely, suppose  $\B_0(H/H_1)=0$.  Then, by the exact sequence (\ref{sequence}), 
$H_1 \cap H'=\langle  H_1 \cap K(H) \rangle$.
Let  $(a, \xi(a)^{-1}) \in Z \cap  (H' \times N')$. Since  $a \in H_1\cap H'=\langle  H_1 \cap K(H) \rangle$,  we deduce that
$a=\prod_{i=1}^k [h_i, h_i'], ~h_i, h'_i \in H$ and $ [h_i, h_i'] \in H_1$.
Then  $\xi([h_i,h'_i]) \in  N_1\cap K(N)$ and 
$$(a, \xi(a)^{-1})=\prod_{i=1}^k\Big([h_i,h'_i], \xi([h_i,h'_i])^{-1}\Big) \in \langle Z \cap  (K(H)\times K(N))\rangle.$$
Hence,  $$ Z \cap  (H' \times N')=\langle Z \cap  (K(H)\times K(N))\rangle.$$
Our result follows.
\\

$(ii)$ 
Suppose $\B_0(G)=0$. Then   $Z\cap(H' \times N')=\langle Z \cap( K(H) \times K(N))\rangle $.
Let $a \in H_1 \cap H'$.
Then $a=\prod_{i=1}^k [h_i, h_i'], ~h_i, h'_i \in H$ and 
$$\xi(a)=\xi( \prod_{i=1}^k [h_i, h_i'])= \prod_{i=1}^k \eta([h_i, h_i']) = \prod_{i=1}^k [\eta(h_i), \eta(h_i')] \in N_1\cap N'.$$
Therefore, $$(a, \xi(a)^{-1}) \in Z\cap(H' \times N')=\langle Z \cap( K(H) \times K(N))\rangle.$$ 
Hence $a \in \langle H_1 \cap K(H)  \rangle$ and so  $ H_1 \cap H'=\langle H_1 \cap K(H)  \rangle$.  Thus $\B_0(H/H_1)=0$  follows by the exact sequence  (\ref{sequence}).
 \\

Conversely, if  $\B_0(H/H_1)=0$, then, by (\ref{sequence}),  $H_1 \cap H'=\langle  H_1 \cap K(H) \rangle$. So,  for $(a, \xi(a)^{-1}) \in Z \cap  (H' \times N')$,
$$a =\prod_{i=1}^s[h_i,h_i'], $$ for $h_i, h_i' \in H$ and 
$[h_i, h_i' ] \in H_1$. Thus $\xi([h_i, h_i' ])= [\eta(h_i), \eta(h_i') ]  \in N_1  \cap  K(N)$ and
 $$\xi(a)=\xi(\prod_{i=1}^s[h_i,h_i'])=\prod_{i=1}^s[\eta(h_i),\eta(h_i')] \in \langle  N_1 \cap K(N) \rangle.$$ Therefore, $(a, \xi(a)^{-1}) \in \langle Z \cap( K(H) \times K(N))\rangle$.
\end{proof}

\begin{thm}
Suppose $H, N$ are groups such that $\xi: H \to N$ is a homomorphism and $\xi|_{H_1}: H_1 \to N_1$ is an isomorphism. Let $G$ be a central product of  $H$ and $N$ via the map $\xi|_{H_1}$. 
Suppose  $\B_0(H)=\B_0(N)=0$. 
\begin{enumerate}[label=(\roman*)]
\item If $H'=K(H)$,  then $\B_0(G)=0$. 
\item Suppose $\xi: H \to N$ is an isomorphism. If $N'=K(N)$ or $\B_0(N/N_1)=0$, then $\B_0(G)=0$.
\end{enumerate}
\end{thm}

\begin{proof}
By Theorem \ref{exactsequence_Bogo}, we have the following exact sequence
\begin{align*}
 0 \to  \Hom\big(\frac{Z\cap(H' \times N')}{\langle Z \cap( K(H) \times K(N))\rangle} , \mathbb{Q}/\mathbb{Z}\big)  \xrightarrow{\tra}  \B_0(G)  
\xrightarrow{\inf} \B_0(H) \times \B_0(N)=0,  
\end{align*}
which implies that  $$\B_0(G) \cong \frac{Z\cap(H' \times N')}{\langle Z \cap( K(H) \times K(N))\rangle}.$$\\

$(i)$   Let $(a, \xi(a)^{-1}) \in H' \times N'$ for $a \in H_1$. 
If $H'=K(H)$, then $a=[h_1,h_2]$, for some $h_1,h_2 \in H$. So $\xi(a)=[\xi(h_1), \xi(h_2)] \in K(N)$ implies $\B_0(G)=0.$\\

$(ii)$ 
  Let $(a, \xi(a)^{-1}) \in H' \times N'$ for $a \in H_1$. 
If $N'=K(N)$, then $\xi(a) =[n_1,n_2]$, for $n_1, n_2 \in N$. So
$a=\xi^{-1}([n_1, n_2])=[\xi^{-1}(n_1), \xi^{-1}(n_2]) \in H_1 \cap K(H)$. This completes the proof.

If  $\B_0(N/N_1)=0$, then, by (\ref{sequence}), $N_1 \cap N'=\langle  N_1 \cap K(N) \rangle$. For   $n_i, n_i' \in  N$ and $[n_i,n_i'] \in N_1$, $$\xi(a) =\prod_{i=1}^t[n_i,n_i'], $$ so
$$a=\xi^{-1}(\prod_{i=1}^t[n_i,n_i']=\prod_{i=1}^s[\xi^{-1}(n_i),\xi^{-1}(n_i')] \in \langle  H_1 \cap K(H) \rangle.$$ Therefore, $(a, \xi(a)^{-1}) \in \langle Z \cap( K(H) \times K(N))\rangle$.

\end{proof}

Now we consider generalized discrete Heisenberg groups.
\begin{example}
Let $r \in \mathbb{N}$.  For positive integers $d_1, d_2, \ldots, d_n$ such that  $d_1|d_2|\cdots|d_n|r$, we define the generalized discrete Heisenberg group $H^{d_1|d_2|\cdots|d_n|r}_{2n+1} (\bZ/r\bZ)$ as the set $(\bZ/r\bZ)^{2n+1}$ with multiplication given by
\[
\begin{array}{l} 
(a, b_1, \ldots, b_n, c_1, \ldots, c_n) (a', b'_1, b'_2, \ldots, b'_n, c'_1, c'_2, \ldots, c'_n) \\ 
= (a+a'+ (\sum_{i = 1}^n d_ib'_i c_i), b_1 + b'_1, \ldots, b_n + b'_n, c_1 + c'_1, \ldots, c_n + c'_n).
\end{array}
\]
\end{example}

\begin{thm}\label{Heisenberg}
$\B_0(H^{d_1|d_2|\cdots|d_n|r}_{2n+1} (\bZ/r\bZ))=0$.
\end{thm}
\begin{proof}
Let $G=H^{d_1|d_2|\cdots|d_n|r}_{2n+1} (\bZ/r\bZ)$.  
It is easy to observe that $G$ has nilpotency class $2$ and 
\[
G=\langle x_i, y_i, z,    1 \leq i \leq n\mid  [x_i, y_i]=z^{d_i}, [z,x_i]=[z,y_i]=1, x^r=y^r=z^r=1\rangle.
\] 
First we show that $\B_0(H^{d_1|r}_{3} (\bZ/r\bZ))=0$.
By Lemma \ref{lemma3},   $[G ,G^{\phi}]$ is generated by $[x, y^\phi], [z, x^\phi], $ and $[z,y^\phi]$. By  Lemma \ref{lemma1}, $[G,G^\phi]$ is abelian and so every $w\in [G,G^\phi]$ can be written  as 
$$w=[x_1,y_1^\phi]^{m_1} ~( \text{mod } M_0^*(G)).$$
Since $[x_1, y_1]$ has order $r/d_1$ and $w \in M^*(G)$,  $r/d_1$ divides $m_1$.
Hence $$M^*(G)=\langle [x_1, y_1^\phi]^{r/d_1} \rangle M_0^*(G).$$
By Lemma \ref{lemma2}, it follows that
$$[x_1, y_1^\phi]^{r/d_1}=[x_1^{r/d_1}, y_1^\phi]=1 \;\; (\text{mod} \; M_0^*(G)).$$
Hence  $\B_0(H^{d_1|r}_{3} (\bZ/r\bZ))=0$.

For $n >1$,  $G$ is a  central product of normal subgroups $H=H^{d_1|r}_{3} (\bZ/r\bZ)$ and $N=H^{d_2|\cdots|d_n|r}_{2n-1} (\bZ/r\bZ)$, where $H_1=\langle
z \rangle$ and  $N_1=\langle z \rangle $ are isomorphic subgroups of $H$ and $N$ respectively.  
Now we use induction on $n$. 
Here $H'=\langle z^{d_1} \rangle$ and  $K'=\langle   z^{d_2}   \rangle$. For $s \in \mathbb{N}$,
$$z^{sd_1}=[x_1^s, y_1], \;\;    z^{sd_2}=[x_2^s, y_2].$$
Therefore $H'=K(H),N'=K(N)$.  
Since $\B_0(H)=\B_0(N)=0$, by the exact sequence (\ref{sequence}), we deduce that
$$\B_0(G)\cong  \frac{Z\cap(H' \times N')}{\langle Z \cap( K(H) \times K(N))\rangle}=0.$$

\end{proof}

\section{Triviality of $\B_0(G)$ for groups $G$ of order $p^6$}\label{Bogop6}
Finite presentations for groups of order dividing $p^6,$ for odd prime $p$, are given in \cite{James}. 
Here we recall some notation from \cite{James}.
By $\nu$ we denote   the smallest positive integer which is a non-quadratic residue (mod $p$), and by $g$ we  denote the smallest positive integer which is a primitive root (mod $p$). Relations of the form $[\alpha, \beta]= 1$ for generators $\alpha$ and  $\beta$ are omitted in the presentations of the groups. 
 For an element $\alpha_{i+1}$ of a finite $p$-group $G$, by
$\alpha_{i+1}^{\left(p\right)}$, we mean  $\alpha_{i+1}^p \alpha_{i+2}^{p \choose 2} \cdots \alpha_{i+k}^{p \choose k} \cdots \alpha_{i+p}$, where $\alpha_{i+2},...,\alpha_{i+p}$ are suitably defined elements of $G$. 
Now we are ready to prove Theorem \ref{groupsp6}. 
We use  Lemmas \ref{lemma1}, \ref{lemma2}  and  \ref{lemma3} heavily in the following proof, without further reference.\\

\noindent\textbf{Proof of Theorem \ref{groupsp6}.}

\begin{proof}
It is enough to show  that  $\B_0(G)=0$ for  $G \in \{\Phi_{15}, \Phi_{28}, \Phi_{29}   \}$,  as the result then follows from \cite[Theorem 1.1]{chen}. Moravec \cite{Mora3} proved that $\B_0(G)$  is invariant up to isoclinism, so  it is enough to choose one group $G$  from each isoclinism class. 

For $x,y \in G$, 
\begin{eqnarray*}
&&[xy,z]=[x,z][x,z,y][y,z], \;\;\;\;  [x,yz]=[x,z][x,y][x,y,z],\\
&&[x^{-1},y]=[x,y,x^{-1}]^{-1}[x,y]^{-1}.
\end{eqnarray*}
We use these equalities in the following computations.\\

\noindent \textbf{Group $G$ in $\Phi_{15}$:}  
We consider
\begin{eqnarray*}
G =\Phi_{15}(1^6) &=&\langle \alpha_i, \beta_1, \beta_2\;\;1 \leq i \leq 4 \mid [\alpha_1, \alpha_2]=[\alpha_3, \alpha_4]=\beta_1,   [\alpha_1, \alpha_3]=\beta_2, \\
&&[\alpha_2, \alpha_4]=\beta_2^g, 
\alpha_i^p=\beta_1^p=\beta_2^p=1\rangle.
\end{eqnarray*}
Since $G'$ is abelian,  $[G,G^\phi]$ has nilpotency class  at most  $2$.
Hence,   every  $w \in M^*(G)$ can be written as 
$$w =[\alpha_1, \alpha_2^\phi]^{m_1}[\alpha_3, \alpha_4^\phi]^{m_2}  [\alpha_1, \alpha_3^\phi]^{m_3}[\alpha_2, \alpha_4^\phi]^{m_4} \; \; (\text{mod } M_0^*(G)).$$
Now 
$\beta_1^{m_1+m_2} \beta_2^{m_3+gm_4}=1$ imply that
 $(m_1+m_2)$ and $(m_3+gm_4)$ are divisible by $p$.
Observe that, (mod $M_0^*(G)$), we have the following:
\begin{eqnarray*}
[\alpha_3, \alpha_4^\phi]^{p}=[\alpha_3^p, \alpha_4^\phi]=1, \;\;  [\alpha_1, \alpha_3^\phi]^{p}=[\alpha_1^p, \alpha_3^\phi]=1.
\end{eqnarray*}  
Thus, 
$$w =([\alpha_1, \alpha_2^\phi][\alpha_3, \alpha_4^\phi]^{-1})^{m_1}  ([\alpha_1, \alpha_3^\phi]^{-g}[\alpha_2, \alpha_4^\phi])^{m_4} \; \; (\text{mod } M_0^*(G)).$$

Now our aim is to show that $[\alpha_1, \alpha_2^\phi][\alpha_3, \alpha_4^\phi]^{-1}$ and $[\alpha_1, \alpha_3^\phi]^{-g}[\alpha_2, \alpha_4^\phi]$ are in  $M_0^*(G)$.
Since $[G', G^\phi] \in M_0^*(G)$, we have the following identities $(\mathrm{mod } ~M_0^*(G))$ :
\begin{eqnarray*}
&&[\alpha_3^{-1}, \alpha_4^\phi]=   [\alpha_3, \alpha_4^\phi]^{-1},  \\
&& [\alpha_1^{-g} , \alpha_3^\phi] = [\alpha_1, \alpha_3^\phi]^{-g},\\
&&[\alpha_1\alpha_3^{-1}, (\alpha_2\alpha_4)^\phi]=[\alpha_1, (\alpha_2\alpha_4)^\phi][ \alpha_3^{-1}, (\alpha_2\alpha_4)^\phi]
=[\alpha_1, \alpha_2^\phi][\alpha_3, \alpha_4^\phi]^{-1}, \\
&&[\alpha_1^{-g}\alpha_2, (\alpha_3\alpha_4)^{\phi}]=[\alpha_1^{-g}, (\alpha_3\alpha_4)^{\phi}][\alpha_2, (\alpha_3\alpha_4)^{\phi}]=  [\alpha_1, \alpha_3^{\phi}]^{-g}  [\alpha_2, \alpha_4^\phi],\\
&&[\alpha_1\alpha_3^{-1}, \alpha_2\alpha_4]=[\alpha_1, \alpha_2\alpha_4][ \alpha_3^{-1}, \alpha_2\alpha_4]
=[\alpha_1, \alpha_2][\alpha_3, \alpha_4]^{-1}=1,\\
&& [\alpha_1^{-g}\alpha_2, \alpha_3\alpha_4]=[\alpha_1^{-g}, \alpha_3\alpha_4][\alpha_2, \alpha_3\alpha_4]=  [\alpha_1, \alpha_3]^{-g}  [\alpha_2, \alpha_4]=1.
\end{eqnarray*}
From the above equalities, 
$$[\alpha_1\alpha_3^{-1}, (\alpha_2\alpha_4)^\phi],  [\alpha_1^{-g}\alpha_2, (\alpha_3\alpha_4)^{\phi}]\in  M_0^*(G).$$
Therefore  $[\alpha_1, \alpha_2^\phi][\alpha_3, \alpha_4^\phi]^{-1}$ and $[\alpha_1, \alpha_3^\phi]^{-g}[\alpha_2, \alpha_4^\phi]$ are in  $M_0^*(G)$, which implies that 
$$\B_0(\Phi_{15}(1^6))=0.$$

\noindent \textbf{Groups $G$ in $\Phi_{28}, \Phi_{29}$ ($p>3$):}\\

For $p>3$, consider the following  groups:
\begin{eqnarray*}
G =\Phi_{28}(222) &=&\langle \alpha, \alpha_i\;\;1 \leq i \leq 4 \mid [\alpha_1, \alpha]=\alpha_2,  [\alpha_2, \alpha]=\alpha_3,  [\alpha_3, \alpha]=[\alpha_1,\alpha_2]=\alpha_4, \\
&& \alpha_3=\alpha_1^p\alpha_2^{\frac{p(p-1)}{2}}, \alpha_4=\alpha_2^p, 
\alpha^{p^2}=\alpha_3^p=\alpha_4^p=1\rangle,
\end{eqnarray*}
and
\begin{eqnarray*}
G =\Phi_{29}(222) &=&\langle \alpha, \alpha_i\;\;1 \leq i \leq 4 \mid [\alpha_1, \alpha]=\alpha_2,  [\alpha_2, \alpha]=\alpha_3,  [\alpha_3, \alpha]=[\alpha_1,\alpha_2]=\alpha_4, \\
&& \alpha_3^{\nu}=\alpha_1^p\alpha_2^{\frac{p(p-1)}{2}}, \alpha_4^{\nu}=\alpha_2^p, 
\alpha^{p^2}=\alpha_3^p=\alpha_4^p=1\rangle.
\end{eqnarray*}
Since $G$ has nilpotency class $4$, $\tau(G)$ has nilpotency class at most $5$. 
Hence, for both  groups, we have the following identities $(\text{mod}~  M_0^*(G))$:
	\begin{align}\label{relation}
&  [\alpha_3^p, \alpha^\phi] = [\alpha_3, \alpha^\phi]^p =1, \nonumber \\
&   [\alpha_2^p, \alpha_1^\phi]= [\alpha_2, \alpha_1^\phi]^p=1, \nonumber\\
& [\alpha_2^p, \alpha^\phi]= [\alpha_2, \alpha^\phi]^p [\alpha_2,\alpha, \alpha_2^\phi]^{p \choose 2}= [\alpha_2, \alpha^\phi]^p [\alpha_3^{p\choose 2}, \alpha_2^\phi]=[\alpha_2, \alpha^\phi]^p=1,
\nonumber \\
& [\alpha_3^{-1},\alpha^\phi]=[\alpha_3,\alpha, (\alpha_3^{-1})^\phi]^{-1}[\alpha_3, \alpha^\phi]^{-1}=[\alpha_3, \alpha^\phi]^{-1},\tag{5.1} 
\\
& [\alpha_2^{-1},\alpha^\phi]=[\alpha_2,\alpha, (\alpha_2^{-1})^\phi]^{-1}[\alpha_2, \alpha^\phi]^{-1}=[\alpha_2, \alpha^\phi]^{-1},\nonumber \\
& [\alpha^p, \alpha_2^\phi]=[\alpha, \alpha_2^\phi]^p[\alpha_3^{-1}, \alpha^\phi]^{p\choose 2}[\alpha_3^{-1}, \alpha,\alpha]^{p\choose 3}=[\alpha, \alpha_2^\phi]^p=[\alpha_2, \alpha^\phi]^{-p}=1,
\nonumber \\
& [\alpha_1^{p^2}, \alpha^\phi]=[\alpha_1,\alpha^\phi]^{p^2}[\alpha_2,\alpha^\phi]^{p^2 \choose 2}=[\alpha_1,\alpha^\phi]^{p^2}=1,
\nonumber \\
&[\alpha^p, \alpha_1^\phi]=[\alpha, \alpha_1^\phi]^p[\alpha_2^{-1}, \alpha^\phi]^{p\choose 2}[\alpha_2^{-1}, \alpha,\alpha^\phi]^{p\choose 3}=[\alpha, \alpha_1^\phi]^p[\alpha_2, \alpha^\phi]^{-{p\choose 2}}[\alpha_3, \alpha^\phi]^{-{p\choose 3}}=[\alpha, \alpha_1^\phi]^p,
\nonumber \\
&[\alpha_1^p, \alpha^\phi]=[\alpha_1, \alpha^\phi]^p [\alpha_1,\alpha, \alpha_1^\phi]^{p \choose 2} [\alpha_2,\alpha_1, \alpha_1^\phi]^{p \choose 3}=[\alpha_1, \alpha^\phi]^p [\alpha_2^{p \choose 2}, \alpha_1^\phi] [\alpha_4^{p \choose 3}, \alpha_1^\phi]
=[\alpha_1, \alpha^\phi]^p, \nonumber\\
&[\alpha_1,\alpha^\phi]^{-p}=[\alpha^\phi, \alpha_1]^p=[\alpha,(\alpha_1^p)^\phi]=[\alpha^p, \alpha_1^\phi].\nonumber
\end{align}
Since $G'$ is abelian, $[G,G^\phi]$ has nilpotency class at most $2$. Therefore every $w \in M^*(G)$ can be written as 
\begin{equation}\label{eq1}\tag{5.2}
w =[\alpha_2, \alpha^\phi]^{m_1}[\alpha_1, \alpha^\phi]^{m_2}  [\alpha_3, \alpha^\phi]^{m_3}[\alpha_1, \alpha_2^\phi]^{m_4} \; \; ( \text{mod } M_0^*(G)),
\end{equation}
for some integers $m_1,m_2,m_3,m_4$.
It is easy to see that  $[\alpha_2^{p \choose 2}, \alpha_1^p]=1$, so  $(\alpha_1^p\alpha_2^{p \choose 2})^{m}=\alpha_1^{mp}\alpha_2^{m{p \choose 2}}$ for every positive integer $m$.

\bigskip

\textbf{Now consider \textbf{$G=\Phi_{28}(222)$}.}  Since $w \in M^*(G)$, by (\ref{eq1}) 
$$\alpha_3^{m_1}\alpha_2^{m_2}\alpha_4^{m_3+m_4}=\alpha_1^{pm_1}\alpha_2^{{p\choose 2} m_1+m_2+p(m_3+m_4)}=1.$$
Therefore, $p$ divides $m_1$  and $p^2$ divides $m_2+p(m_3+m_4)$. Since 
$[\alpha_1, \alpha^\phi]^{p^2}, [\alpha_2,\alpha^\phi]^p\in  M_0^*(G)$, 
by (\ref{eq1})
we have   $( \text{mod } M_0^*(G))$,
\begin{equation}\label{1}\tag{5.3}
\begin{split}
w&=[\alpha_1, \alpha^\phi]^{-p(m_3+m_4)}  [\alpha_3, \alpha^\phi]^{m_3}[\alpha_1, \alpha_2^\phi]^{m_4} \; \; \\
&= ([\alpha_3,\alpha^\phi][\alpha_1, \alpha^\phi]^{-p})^{m_3} ([\alpha_1, \alpha_2^\phi] [\alpha_1, \alpha^\phi]^{-p})^{m_4}  \nonumber\\
&=   ([\alpha_3,\alpha^\phi][\alpha, (\alpha_1^p)^\phi])^{m_3} ([\alpha_1, \alpha_2^\phi] [\alpha^p, \alpha_1^\phi])^{m_4}.
\end{split}
\end{equation}
To prove $\B_0(G)=0$, it is enough to  show that 
$[\alpha_3, \alpha^\phi][\alpha, (\alpha_1^p)^\phi]$ and  $[\alpha_1, \alpha_2^\phi][\alpha^p,\alpha_1^\phi]$ are in $M_0^*(G)$.
Observe that 
$$[\alpha_1^p,\alpha]=[\alpha_1,\alpha]^p[\alpha_2,\alpha_1]^{p\choose 2}=[\alpha_1,\alpha]^p,$$
$$[\alpha_3\alpha, \alpha\alpha_1^p]=[\alpha_3, \alpha\alpha_1^p][\alpha_3, \alpha\alpha_1^p, \alpha] [\alpha, \alpha\alpha_1^p]=[\alpha_3, \alpha][\alpha,\alpha_1^p]=[\alpha_3, \alpha][\alpha,\alpha_1]^p=1,$$
$$[\alpha^p\alpha_1, \alpha_1\alpha_2]=[\alpha^p, \alpha_1\alpha_2][\alpha^p, \alpha_1\alpha_2, \alpha_1][\alpha_1, \alpha_1\alpha_2]=[\alpha^p,\alpha_1][\alpha_1,\alpha_2]=[\alpha,\alpha_1]^p[\alpha_1,\alpha_2]=1.$$ 

Now  using all the relations of   (\ref{relation}),  $(\text{mod}~ M_0^*(G))$,
$$
[\alpha_3\alpha, (\alpha\alpha_1^p)^\phi]=[\alpha_3, (\alpha\alpha_1^p)^\phi][\alpha_3, \alpha\alpha_1^p, \alpha^\phi] [\alpha, (\alpha\alpha_1^p)^\phi]=[\alpha_3, \alpha^\phi][\alpha, (\alpha_1^p)^\phi],$$
\begin{eqnarray*}
[\alpha_1\alpha^p, (\alpha_1\alpha_2)^\phi]&=&[\alpha_1, (\alpha_1\alpha_2)^\phi]
[\alpha_1, (\alpha_1\alpha_2)^\phi, \alpha^p]
[\alpha^p, (\alpha_1\alpha_2)^\phi]\\
&=&   [\alpha_1,\alpha_2^\phi][\alpha_1,\alpha_2, (\alpha^p)^\phi][\alpha^p,\alpha_1^\phi]  =[\alpha_1,\alpha_2^\phi][\alpha^p,\alpha_1^\phi].
\end{eqnarray*}
Therefore,  $$[\alpha_3\alpha, (\alpha\alpha_1^p)^\phi], [\alpha^p\alpha_1, (\alpha_1\alpha_2)^\phi] \in M_0^*(G)$$ implies that 
$$[\alpha_3, \alpha^\phi][\alpha, (\alpha_1^p)^\phi], [\alpha_1, \alpha_2^\phi][\alpha^p,\alpha_1^\phi]\in M_0^*(G)$$ and so, by (\ref{1}), $w \in M_0^*(G)$.
Thus $\B_0(\Phi_{28}(222))=0.$

\bigskip

\textbf{Next consider \textbf{$G=\Phi_{29}(222)$}.} There exists a positive integer $s$ such that $s <p$ and $\nu s=1$ (mod $p$). 
Since $w \in M^*(G)$, by (\ref{eq1}) 
$$\alpha_3^{m_1}\alpha_2^{m_2}\alpha_4^{m_3+m_4}=\alpha_3^{\nu  sm_1}\alpha_2^{m_2}\alpha_4^{\nu s(m_3+m_4)}=\alpha_1^{psm_1}\alpha_2^{{p\choose 2} sm_1+m_2+ps(m_3+m_4)}=1.$$
Therefore, $p$ divides $m_1$  and $p^2$ divides $m_2+ps(m_3+m_4)$. Since 
$[\alpha_1, \alpha^\phi]^{p^2}, [\alpha_2,\alpha^\phi]^p\in  M_0^*(G)$
and using all the relations of  (\ref{relation}), we have  the following $(\text{mod}~  M_0^*(G))$:
\begin{eqnarray*}
&&[\alpha_1^p,\alpha]=[\alpha_1,\alpha]^p[\alpha_2,\alpha_1]^{p\choose 2}=[\alpha_1,\alpha]^p=[\alpha_1,\alpha^p],\\
&&[\alpha^s, (\alpha_1^p)^\phi]=[\alpha,(\alpha_1^p)^\phi]^{s} [\alpha,\alpha_1^p,\alpha^\phi] ^{{s \choose 2}} =[\alpha, (\alpha_1^p)^\phi]^s,
\\
&& [\alpha^p, (\alpha_1^s)^\phi]=[\alpha_1^s, (\alpha^p)^\phi]^{-1}=[\alpha_1,(\alpha^p)^\phi]^{-s}[\alpha_1,\alpha^p, \alpha_1^\phi]^{-{s \choose 2}}= [\alpha^p, \alpha_1^\phi]^s.
\end{eqnarray*}
Therefore, by (\ref{eq1}), we have the following
$ (\text{mod}~  M_0^*(G))$:
\begin{equation}\label{2}\tag{5.4}
\begin{split}
w&=[\alpha_1, \alpha^\phi]^{-ps(m_3+m_4)}  [\alpha_3, \alpha^\phi]^{m_3}[\alpha_1, \alpha_2^\phi]^{m_4} \\
&= ([\alpha_3,\alpha^\phi][\alpha_1, \alpha^\phi]^{-ps})^{m_3} ([\alpha_1, \alpha_2^\phi] [\alpha_1, \alpha^\phi]^{-ps})^{m_4} \nonumber\\
&=   ([\alpha_3,\alpha^\phi][\alpha, (\alpha_1^p)^\phi]^s)^{m_3} ([\alpha_1, \alpha_2^\phi] [\alpha^p, \alpha_1^\phi]^s)^{m_4}   \nonumber\\
&=   ([\alpha_3,\alpha^\phi][\alpha^s, (\alpha_1^p)^\phi])^{m_3} ([\alpha_1, \alpha_2^\phi] [\alpha^p, (\alpha_1^s)^\phi])^{m_4}.
\end{split}
\end{equation}
Hence
\begin{eqnarray*}
&&[\alpha^s,\alpha_1]^p=[\alpha,\alpha_1]^{ps} [\alpha_2^{-1},\alpha] ^{p{s \choose 2}} =[\alpha,\alpha_1]^{ps},
\\
&&[\alpha_1^s,\alpha]^p=[\alpha_1,\alpha]^{ps} [\alpha_2,\alpha_1] ^{p{s \choose 2}} =[\alpha_1,\alpha]^{ps},
\\
&&[\alpha_1^p,\alpha^s]=[\alpha_1, \alpha^s]^p[\alpha_1, \alpha^s, \alpha_1]^{p\choose 2}=[\alpha_1, \alpha]^{sp}[[\alpha_1, \alpha]^{s}, \alpha_1]^{p\choose 2}=[\alpha_1, \alpha]^{sp},
\\
&&[\alpha^p,\alpha_1^s]=[\alpha,\alpha_1^s]^p[\alpha,\alpha_1^s, \alpha]^{p\choose 2}=[\alpha,\alpha_1]^{sp}[[\alpha,\alpha_1]^{s}, \alpha]^{p\choose 2}=[\alpha,\alpha_1]^{sp}
\\
&&[\alpha_3\alpha^s, \alpha\alpha_1^p]=[\alpha_3, \alpha\alpha_1^p][\alpha_3, \alpha\alpha_1^p, \alpha^s] [\alpha^s, \alpha\alpha_1^p]=[\alpha_3, \alpha][\alpha^s,\alpha_1^p]=[\alpha_3, \alpha][\alpha,\alpha_1]^{sp}\\
&& \hspace{2.2cm}  =\alpha_4\alpha_2^{-sp}=\alpha_4\alpha_4^{-\nu s}=\alpha_4\alpha_4^{-1}=1,\\
&&[\alpha_1\alpha^p, \alpha_1^s\alpha_2]=[\alpha_1, \alpha_1^s\alpha_2]
[\alpha_1, \alpha_1^s\alpha_2, \alpha^p]
[\alpha^p, \alpha_1^s\alpha_2]=  [\alpha_1,\alpha_2][\alpha_1,\alpha_2, \alpha^p][\alpha^p,\alpha_1]  \\
&& \hspace{2.2cm} =[\alpha_1,\alpha_2][\alpha^p,\alpha_1^s]=\alpha_4\alpha_2^{-sp}=\alpha_4\alpha_4^{-\nu s}=\alpha_4\alpha_4^{-1}=1.
\end{eqnarray*}

Now $(\text{mod} ~M_0^*(G))$,  
$$
[\alpha_3\alpha^s, (\alpha\alpha_1^p)^\phi]=[\alpha_3, (\alpha\alpha_1^p)^\phi][\alpha_3, \alpha\alpha_1^p, (\alpha^s)^\phi] [\alpha^s, (\alpha\alpha_1^p)^\phi]=[\alpha_3, \alpha^\phi][\alpha^s, (\alpha_1^p)^\phi],$$
\begin{eqnarray*}
[\alpha_1\alpha^p, (\alpha_1^s\alpha_2)^\phi]&=&[\alpha_1, (\alpha_1^s\alpha_2)^\phi]
[\alpha_1, (\alpha_1^s\alpha_2)^\phi, \alpha^p]
[\alpha^p, (\alpha_1^s\alpha_2)^\phi]\\
&=&   [\alpha_1,\alpha_2^\phi][\alpha_1,\alpha_2, (\alpha^p)^{\phi}][\alpha^p,\alpha_1^\phi]  =[\alpha_1,\alpha_2^\phi][\alpha^p,(\alpha_1^s)^\phi].
\end{eqnarray*}
Therefore,    
$[\alpha_3\alpha^s, \alpha\alpha_1^p]=[\alpha_1\alpha^p, \alpha_1^s\alpha_2]=1$ implies that   
$$[\alpha_3, \alpha^\phi][\alpha^s, (\alpha_1^p)^\phi], [\alpha_1,\alpha_2^\phi][\alpha^p,(\alpha_1^s)^\phi] \in M_0^*(G).$$ 
Hence  by (\ref{2}), $w \in M_0^*(G)$ and 
$\B_0(\Phi_{29}(222))=0.$\\

\noindent \textbf{Groups $G$ in $\Phi_{28}, \Phi_{29}$ ($p=3$):}\\

Let $p=3$, $\nu=2$. We consider the following groups:
\begin{eqnarray*}
	G=\Phi_{28}(222) &=&\langle \alpha, \alpha_i\;\;1 \leq i \leq 4 \mid [\alpha_1, \alpha]=\alpha_2,  [\alpha_2, \alpha]=\alpha_3,  [\alpha_3, \alpha]=[\alpha_1,\alpha_2]=\alpha_4, \\
	&& \alpha_1^3=\alpha_2^{3}=1, 
	\alpha^{9}=\alpha_3^3=\alpha_4^3=1\rangle,\\
	G=\Phi_{29}(222) &=&\langle \alpha, \alpha_i\;\;1 \leq i \leq 4 \mid [\alpha_1, \alpha]=\alpha_2,  [\alpha_2, \alpha]=\alpha_3,  [\alpha_3, \alpha]=[\alpha_1,\alpha_2]=\alpha_4, \\
	&& \alpha_3=\alpha_1^3\alpha_2^{3}, \alpha_4=\alpha_2^3, 
	\alpha^{9}=\alpha_3^3=\alpha_4^3=1\rangle.
\end{eqnarray*}
We denote the $i$-th group  of order $3^6$  in the {\sc SmallGroups} library  \cite{besche2002millennium}  by $3^6\# i$.
Our computations in GAP \cite{gap2018groups} show that $\Phi_{28}(222)\cong 3^6\# 68,~\Phi_{29}(222)\cong  3^6\# 75$, and both groups have  trivial Bogomolov multiplier.
\end{proof}

\noindent
{\bf Acknowledgement:} 
The author acknowledges the support received from  C.V. Raman postdoctoral fellowship (R(IA)/CVR-PDF/2020/2700).
The author is thankful to Prof. E. K. Narayanan and Prof. Pooja Singla for their encouragement and helpful comments. The author also thanks the referees for careful reading and valuable comments which  improved the article.

\bibliographystyle{amsplain}
\bibliography{Bogomolov}

\end{document}